\documentclass{amsart}%
\usepackage{amsfonts}
\usepackage{amsmath}
\usepackage{amssymb}
\usepackage{graphicx}
\usepackage[latin1]{inputenc}
\usepackage[english]{babel}
\usepackage{hyperref}%
\newtheorem{theorem}{Theorem}[section]
\theoremstyle{plain}

\newtheorem{conjecture}[theorem]{Conjecture}

\newtheorem{definition}[theorem]{Definition}
\newtheorem{example}[theorem]{Example}

\newtheorem{lemma}[theorem]{Lemma}

\newtheorem{proposition}[theorem]{Proposition}
\newtheorem{remark}[theorem]{Remark}

\numberwithin{equation}{section}
\begin{document}
\title[Heat Traces and $p-$adic Spectral Zeta Functions ]{Heat Traces and Spectral Zeta Functions for $p-$adic Laplacians}
\author{L. F. Chacón-Cortés}
\address{Pontificia Universidad Javeriana \\
Departamento de Matem\'aticas\\
Facultad de Ciencias\\
Cra. 7 No. 43-82, Bogot\'a, Colombia}
\email{leonardo.chacon@javeriana.edu.co}
\author{W. A. Zú\~niga-Galindo}
\address{Centro de Investigación y de Estudios Avanzados, Departamento de Matemáticas,
Unidad Querétaro, Libramiento Norponiente \#2000, Fracc. Real de Juriquilla.
C.P. 76230, Querétaro, QRO, México}
\email{wazuniga@math.cinvestav.edu.mx}
\thanks{The second author was partially supported by the ABACUS project, EDOMEX-2011-CO1-165873.}

\begin{abstract}
In this article we initiate the study of the heat traces and spectral zeta
functions for certain $p$-adic Laplacians. We show that the heat traces are
given by $p$-adic integrals of Laplace type, and that the spectral zeta
functions are $p$-adic integrals of Igusa-type. We find good estimates for the
behaviour of the heat traces when the time tends to infinity, and for the
asymptotics of the function counting the eigenvalues less than or equal to a
given quantity.

\end{abstract}
\maketitle

\section{Introduction}

The $p$-adic heat equation is defined as%
\begin{equation}
\frac{\partial u\left(  x,t\right)  }{\partial t}+D^{\beta}u\left(
x,t\right)  =0\text{, }x\in\mathbb{Q}_{p}\text{, }t\geq
0\label{heat_eqution_intro}%
\end{equation}
where
\[
\left(  D^{\beta}\varphi\right)  \left(  x\right)  =\mathcal{F}_{\xi
\rightarrow x}^{-1}\left(  \left\vert \xi\right\vert _{p}^{\beta}%
\mathcal{F}_{x\rightarrow\xi}\varphi\right)  ,\beta>0,
\]
is the\textbf{\ }Vladimirov operator (a $p$-adic Laplacian), and $\mathcal{F}$
denotes the $p$-adic Fou\-rier transform. This equation is the $p$-adic
counterpart of the classical fractional heat equation, which describes
particle performing a random motion (the fractional Brownian motion), a
`similar' statement is valid for the $p$-adic heat equation. More precisely,
the fundamental solution of (\ref{heat_eqution_intro}) is the transition
density of a bounded right-continuous Markov process without second kind
discontinuities. The family of non-Archimedean heat-type equations is very
large and it has deep connections with mathematical physics. For instance, in
\cite{Av-4}-\cite{Av-5}, Avetisov et al. introduced a new class of models for
complex systems based on $p$-adic analysis. From a mathematical point view, in
these models the time-evolution of a complex system is described by a $p$-adic
master equation (a parabolic-type pseudodifferential equation) which controls
the time-evolution of a transition function of a Markov process on an
ultrametric space. The simplest type of master equation is the one-dimensional
$p$-adic heat equation. This equation was introduced in the book of
Vladimirov, Volovich and Zelenov \cite{V-V-Z}. It is worth to mention here,
that the $p$-adic heat equation also appeared in certain works connected with
the Riemann hypothesis \cite{Leic}. In recent years the non-Archimedean
heat-type equations and their associated Markov processes have been studied
intensively, see e.g. \cite{Koch}, \cite{V-V-Z}, \cite{Casas-Zuniga},
\cite{Ch-Z}, \cite{Ch-Z-1}, \cite{Dra-Kh-K-V}, \cite{Ro-Zuniga}, \cite{T-Z},
\cite{Zuniga1}, \cite{Zuniga2} and the references therein.

The connections between the Archimedean heat equations with number theory and
geometry are well-known and deep. Let us mention here, the connection with the
Riemann zeta function which drives naturally to trace-type formulas, see e.g.
\cite[and the references therein]{Arendt et al}, and the connection with the
Atiyah-Singer index Theorem, see e.g. \cite[and the references therein]%
{Gilkey}. The study of non-Archimedean counterparts of the above mentioned
matters is quite relevant, specially taking into account that the Connes and
Deninger programs to attack the Riemann hypothesis drive naturally to these
matters, see e.g. \cite{Connes}, \cite{Denninger}, \cite{Leic-1} and the
references therein. For instance several types of $p$-adic trace formulas have
been studied, see e.g. \cite{Alberverio et al}, \cite{Burnol}, \cite{Yasuda}%
\ and the references therein.

Nowadays there is no a theory of pseudodifferential operators over $p$-adic
ma\-nifolds comparable to the classical theory, see e.g. \cite{Shubin} and the
references therein. The $n$-dimensional unit ball is the simplest $p$-adic
compact manifold possible. From a topological point of view this ball is a
fractal, more precisely, it is topologically equivalent to a Cantor-like
subset of the real line, see e.g. \cite{A-K-S}, \cite{V-V-Z}. Currently, there
is a lot interest on spectral zeta functions attached to fractals see e.g.
\cite{Lal-Lapidius}, \cite{S-T}.

In this article we initiate the study of heat traces and spectral zeta
functions attached to certain $p$-adic Laplacians, denoted as $\boldsymbol{A}%
_{\beta}$, which are generalizations of the $p$-adic Laplacians introduced by
the authors in \cite{Ch-Z}, see also \cite{Ch-Z-1}. By using an approach
inspired on the work of Minakshisundaram and Pleijel, see \cite{Minak}%
-\cite{Minak-P}, we find a formula for the trace of the semigroup
$e^{-t\boldsymbol{A}_{\beta}}$ acting on the space of square integrable
functions supported on the unit ball with average zero, see Theorem
\ref{Theorem2A}. The trace of $e^{-t\boldsymbol{A}_{\beta}}$ is a $p$-adic
oscillatory integral of Laplace-type, we do not know the exact asymptotics of
this integral as $t$ tends to infinity, however, we obtain a good estimation
for its behavior at infinity, see Theorem \ref{Theorem2A} (ii). Several
unexpected mathematical situations occur in the $p$-adic setting. For
instance, the spectral zeta functions are $p$-adic Igusa-type integrals, see
Theorem \ref{Theorem3}. The $p$-adic spectral zeta functions studied here may
have infinitely many poles on the boundary of its domain of holomorphy, then,
to the best of our \ knowledge, the standard Ikehara Tauberian Theorems cannot
be applied to obtain the asymptotic behavior for the function counting the
eigenvalues of $\boldsymbol{A}_{\beta}$ less than or equal to $T\geq0$,
however, we are still able to find good estimates for this function, see
Theorem \ref{Theorem3} and Remark \ref{Nota3}, and Conjecture \ref{Conjecture}%
. The proofs require several results on certain `boundary value problems'
attached to $p$-adic heat equations associated with operators $\boldsymbol{A}%
_{\beta}$, \ see Proposition \ref{Prop1}, Theorem \ref{Theorem1} and
Proposition \ref{Theorem2}. Finally, Let us mention that our results and
techniques are completely different from those presented in \cite{Alberverio
et al}, \cite{Burnol}, \cite{Yasuda}\ .

\section{\label{Section1}Preliminaries}

In this section we fix the notation and collect some basic results on $p$-adic
analysis that we will use through the article. For a detailed exposition on
$p$-adic analysis the reader may consult \cite{A-K-S}, \cite{Taibleson},
\cite{V-V-Z}.

\subsection{The field of $p$-adic numbers}

Along this article $p$ will denote a prime number. The field of $p-$adic
numbers $\mathbb{Q}_{p}$ is defined as the completion of the field of rational
numbers $\mathbb{Q}$ with respect to the $p-$adic norm $|\cdot|_{p}$, which is
defined as
\[
|x|_{p}=%
\begin{cases}
0 & \text{if }x=0\\
p^{-\gamma} & \text{if }x=p^{\gamma}\dfrac{a}{b},
\end{cases}
\]
where $a$ and $b$ are integers coprime with $p$. The integer $\gamma
=ord_{p}(x):=ord(x)$, with $ord(0):=+\infty$, is called the\textit{\ }%
$p-$\textit{adic order of} $x$. We extend the $p-$adic norm to $\mathbb{Q}%
_{p}^{n}$ by taking%
\[
||x||_{p}:=\max_{1\leq i\leq n}|x_{i}|_{p},\qquad\text{for }x=(x_{1}%
,\dots,x_{n})\in\mathbb{Q}_{p}^{n}.
\]
We define $ord(x)=\min_{1\leq i\leq n}\{ord(x_{i})\}$, then $||x||_{p}%
=p^{-ord(x)}$.\ The metric space $\left(  \mathbb{Q}_{p}^{n},||\cdot
||_{p}\right)  $ is a complete ultrametric space. As a topological space
$\mathbb{Q}_{p}$\ is homeomorphic to a Cantor-like subset of the real line,
see e.g. \cite{A-K-S}, \cite{V-V-Z}.

Any $p-$adic number $x\neq0$ has a unique expansion of the form
\[
x=p^{ord(x)}\sum_{j=0}^{\infty}x_{j}p^{j},
\]
where $x_{j}\in\{0,1,2,\dots,p-1\}$ and $x_{0}\neq0$. By using this expansion,
we define \textit{the fractional part }$\{x\}_{p}$\textit{ of }$x\in
\mathbb{Q}_{p}$ as the rational number
\[
\{x\}_{p}=%
\begin{cases}
0 & \text{if }x=0\text{ or }ord(x)\geq0\\
p^{ord(x)}\sum_{j=0}^{-ord(x)-1}x_{j}p^{j} & \text{if }ord(x)<0.
\end{cases}
\]
In addition, any $x\in\mathbb{Q}_{p}^{n}\smallsetminus\left\{  0\right\}  $
can be represented uniquely as $x=p^{ord(x)}v\left(  x\right)  $ where
$\left\Vert v\left(  x\right)  \right\Vert _{p}=1$.

\subsection{Additive characters}

Set $\chi_{p}(y)=\exp(2\pi i\{y\}_{p})$ for $y\in\mathbb{Q}_{p}$. The map
$\chi_{p}(\cdot)$ is an additive character on $\mathbb{Q}_{p}$, i.e. a
continuous map from $\left(  \mathbb{Q}_{p},+\right)  $ into $S$ (the unit
circle considered as multiplicative group) satisfying $\chi_{p}(x_{0}%
+x_{1})=\chi_{p}(x_{0})\chi_{p}(x_{1})$, $x_{0},x_{1}\in\mathbb{Q}_{p}$. \ The
additive characters of $\mathbb{Q}_{p}$ form an Abelian group which is
isomorphic to $\left(  \mathbb{Q}_{p},+\right)  $, the isomorphism is given by
$\xi\rightarrow\chi_{p}(\xi x)$, see e.g. \cite[Section 2.3]{A-K-S}.

\subsection{Topology of $\mathbb{Q}_{p}^{n}$}

For $r\in\mathbb{Z}$, denote by $B_{r}^{n}(a)=\{x\in\mathbb{Q}_{p}
^{n};||x-a||_{p}\leq p^{r}\}$ \textit{the ball of radius }$p^{r}$ \textit{with
center at} $a=(a_{1},\dots,a_{n})\in\mathbb{Q}_{p}^{n}$, and take $B_{r}%
^{n}(0):=B_{r}^{n}$. Note that $B_{r}^{n}(a)=B_{r}(a_{1})\times\cdots\times
B_{r}(a_{n})$, where $B_{r}(a_{i}):=\{x\in\mathbb{Q}_{p};|x_{i}-a_{i}|_{p}\leq
p^{r}\}$ is the one-dimensional ball of radius $p^{r}$ with center at
$a_{i}\in\mathbb{Q}_{p}$. The ball $B_{0}^{n}$ equals the product of $n$
copies of $B_{0}=\mathbb{Z}_{p}$, \textit{the ring of }$p- $\textit{adic
integers}. We also denote by $S_{r}^{n}(a)=\{x\in\mathbb{Q}_{p}^{n}%
;||x-a||_{p}=p^{r}\}$ \textit{the sphere of radius }$p^{r}$ \textit{with
center at} $a=(a_{1},\dots,a_{n})\in\mathbb{Q}_{p}^{n}$, and take $S_{r}%
^{n}(0):=S_{r}^{n}$. We notice that $S_{0}^{1}=\mathbb{Z}_{p}^{\times}$ (the
group of units of $\mathbb{Z}_{p}$), but $\left(  \mathbb{Z}_{p}^{\times
}\right)  ^{n}\subsetneq S_{0}^{n}$. The balls and spheres are both open and
closed subsets in $\mathbb{Q}_{p}^{n}$. In addition, two balls in
$\mathbb{Q}_{p}^{n}$ are either disjoint or one is contained in the other.

As a topological space $\left(  \mathbb{Q}_{p}^{n},||\cdot||_{p}\right)  $ is
totally disconnected, i.e. the only connected \ subsets of $\mathbb{Q}_{p}%
^{n}$ are the empty set and the points. A subset of $\mathbb{Q}_{p}^{n}$ is
compact if and only if it is closed and bounded in $\mathbb{Q}_{p}^{n}$, see
e.g. \cite[Section 1.3]{V-V-Z}, or \cite[Section 1.8]{A-K-S}. The balls and
spheres are compact subsets. Thus $\left(  \mathbb{Q}_{p}^{n},||\cdot
||_{p}\right)  $ is a locally compact topological space.

We will use $\Omega\left(  p^{-r}||x-a||_{p}\right)  $ to denote the
characteristic function of the ball $B_{r}^{n}(a)$. For more general sets, we
will use the notation $1_{A}$ for the characteristic function of a set $A$.

\section{The Bruhat-Schwartz space and the Fourier transform}

A complex-valued function $\varphi$ defined on $\mathbb{Q}_{p}^{n}$ is
\textit{called locally constant} if for any $x\in\mathbb{Q}_{p}^{n}$ there
exist an integer $l(x)\in\mathbb{Z}$ such that%
\begin{equation}
\varphi(x+x^{\prime})=\varphi(x)\text{ for }x^{\prime}\in B_{l(x)}^{n}.
\label{local_constancy}%
\end{equation}

A function $\varphi:\mathbb{Q}_{p}^{n}\rightarrow\mathbb{C}$ is called a
\textit{Bruhat-Schwartz function (or a test function)} if it is locally
constant with compact support. Any test function can be represented as a
linear combination, with complex coefficients, of characteristic functions of
balls. The $\mathbb{C}$-vector space of Bruhat-Schwartz functions is denoted
by $\mathcal{D}(\mathbb{Q}_{p}^{n})$. For $\varphi\in\mathcal{D}%
(\mathbb{Q}_{p}^{n})$, the largest number $l=l(\varphi)$ satisfying
(\ref{local_constancy}) is called \textit{the exponent of local constancy (or
the parameter of constancy) of} $\varphi$.

If $U$ is an open subset of $\mathbb{Q}_{p}^{n}$, $\mathcal{D}(U)$ denotes the
space of test functions with supports contained in $U$, then $\mathcal{D}(U)$
is dense in
\[
L^{\rho}\left(  U\right)  =\left\{  \varphi:U\rightarrow\mathbb{C};\left\Vert
\varphi\right\Vert _{\rho}=\left\{  \int_{U}\left\vert \varphi\left(
x\right)  \right\vert ^{\rho}d^{n}x\right\}  ^{\frac{1}{\rho}}<\infty\right\}
,
\]
where $d^{n}x$ is the Haar measure on $\mathbb{Q}_{p}^{n}$ normalized by the
condition $vol(B_{0}^{n})\allowbreak=1$, for $1\leq\rho<\infty$, see e.g.
\cite[Section 4.3]{A-K-S}.

\subsection{The Fourier transform of test functions}

Given $\xi=(\xi_{1},\dots,\xi_{n})$ and $y=(x_{1},\dots,x_{n})\in
\mathbb{Q}_{p}^{n}$, we set $\xi\cdot x:=\sum_{j=1}^{n}\xi_{j}x_{j}$. The
Fourier transform of $\varphi\in\mathcal{D}(\mathbb{Q}_{p}^{n})$ is defined
as
\[
(\mathcal{F}\varphi)(\xi)=\int_{\mathbb{Q}_{p}^{n}}\chi_{p}(\xi\cdot
x)\varphi(x)d^{n}x\quad\text{for }\xi\in\mathbb{Q}_{p}^{n},
\]
where $d^{n}x$ is the normalized Haar measure on $\mathbb{Q}_{p}^{n}$. The
Fourier transform is a linear isomorphism from $\mathcal{D}(\mathbb{Q}_{p}%
^{n})$ onto itself satisfying $(\mathcal{F}(\mathcal{F}\varphi))(\xi
)=\varphi(-\xi)$, see e.g. \cite[Section 4.8]{A-K-S}. We will also use the
notation $\mathcal{F}_{x\rightarrow\xi}\varphi$ and $\widehat{\varphi}$\ for
the Fourier transform of $\varphi$.

\section{\label{Sect2}A Class of $p-$adic Laplacians}

Take \ $\mathbb{R}_{+}:=\left\{  x\in\mathbb{R};x\geq0\right\}  $, and fix a
function%
\[
A:\mathbb{Q}_{p}^{n}\rightarrow\mathbb{R}_{+}%
\]
satisfying the following properties:

\noindent(i) $A\left(  \xi\right)  $ is a radial function, i.e $A\left(
\xi\right)  =g\left(  \left\Vert \xi\right\Vert _{p}\right)  $ for some
$g:\mathbb{R}_{+}\rightarrow\mathbb{R}_{+}$, by simplicity we use the notation
$A\left(  \xi\right)  =A(\left\Vert \xi\right\Vert _{p})$;

\noindent(ii) there exist constants $C_{0}$, $C_{1}>0$ and $\beta>0$ such
that
\begin{equation}
C_{0}\left\Vert \xi\right\Vert _{p}^{\beta}\leq A(\xi)\leq C_{1}\left\Vert
\xi\right\Vert _{p}^{\beta}\text{, for }x\in\mathbb{Q}_{p}^{n}\text{.}
\label{1w}%
\end{equation}
Taking into account that $\beta$ in (\ref{1w}) is unique, we use the notation
$A_{\beta}(\Vert\xi\Vert_{p})=A(\Vert\xi\Vert_{p})$.

We define the pseudodifferential operator $\boldsymbol{A}_{\beta}$ by
\begin{equation}
\left(  \boldsymbol{A}_{\beta}\varphi\right)  \left(  x\right)  =\mathcal{F}%
_{\xi\rightarrow x}^{-1}\left[  A_{{\beta}}\left(  \xi\right)  \mathcal{F}%
_{x\rightarrow\xi}\varphi\right]  \text{, for }\varphi\in\mathcal{D}\left(
\mathbb{Q}_{p}^{n}\right)  \text{.} \label{pseudo}%
\end{equation}
We will call $A_{\beta}\left(  \xi\right)  $ \textit{the symbol of}
$\boldsymbol{A}_{\beta}$. The operator $\boldsymbol{A}_{\beta}$ extends to an
unbounded and densely defined operator in $L^{2}\left(  \mathbb{Q}_{p}%
^{n}\right)  $ with domain
\begin{equation}
Dom(\boldsymbol{A}_{\beta})=\left\{  \varphi\in L^{2};A_{\beta}(\xi
)\mathcal{F}\varphi\in L^{2}\right\}  . \label{Dom_W}%
\end{equation}
In addition:

\noindent(i) $\left(  \boldsymbol{A}_{\beta},Dom(\boldsymbol{A}_{\beta
})\right)  $ is self-adjoint and positive operator;

\noindent(ii) $-\boldsymbol{A}_{\beta}$ is the infinitesimal generator of a
contraction $C_{0}-$semigroup, cf. \cite[Proposition 3.3]{Ch-Z}.

We attach to operator $\boldsymbol{A}_{\beta}$ the following `heat equation':%

\[
\left\{
\begin{array}
[c]{ll}%
\frac{\partial u(x,t)}{\partial t}+\boldsymbol{A}_{\beta}u(x,t)=0, &
x\in\mathbb{Q}_{p}^{n},\,\,\,t\in\left[  0,\infty\right) \\
& \\
u\left(  x,0\right)  =u_{0}(x), & u_{0}(x)\in Dom(\boldsymbol{A}_{\beta}).
\end{array}
\right.
\]
This initial value problem has a unique solutions given by
\[
u(x,t)=\int\limits_{\mathbb{Q}_{p}^{n}}Z(x-y,t)u_{0}(y)d^{n}y,
\]
where
\[
Z(x,t;\boldsymbol{A}_{\beta}):=Z(x,t)=\int\limits_{\mathbb{Q}_{p}^{n}}\chi
_{p}(-\xi\cdot x)e^{-tA_{\beta}(\xi)}d^{n}\xi,\,\,\,\text{for}%
\,\,\,t>0,\,\,\,x\in\mathbb{Q}_{p}^{n},
\]
cf. \cite[Theorem 6.5]{Ch-Z}. The function $Z(x,t)$ is called the heat kernel
associated with operator $\boldsymbol{A}_{\beta}$.

\subsection{Operators $\boldsymbol{W}_{\alpha}$}

The class of operators $\boldsymbol{A}_{\beta}$ includes the class of
operators $\boldsymbol{W_{\alpha}}$ studied by the authors in \cite{Ch-Z}, see
also \cite{Ch-Z-1}. In addition, most of the results on $\boldsymbol{W_{\alpha
}}$ operators are valid for $\boldsymbol{A}_{\beta}$ operators. We review
briefly the definition of these operators. Fix a function%
\[
w_{\alpha}:\mathbb{Q}_{p}^{n}\rightarrow\mathbb{R}_{+}%
\]
satisfying the following properties:

\noindent(i) $w_{\alpha}\left(  y\right)  $ is a radial i.e. $w_{\alpha
}(y)=w_{\alpha}(\left\Vert y\right\Vert _{p})$;

\noindent(ii) $w_{\alpha}(\left\Vert y\right\Vert _{p})$ is continuous and
increasing function of $\left\Vert y\right\Vert _{p}$;

\noindent(iii) $w_{\alpha}\left(  y\right)  =0$ if and only if $y=0$;

\noindent(iv) there exist constants $C_{0},C_{1}>0$ and $\alpha>n$ such that
\[
C_{0}\left\Vert y\right\Vert _{p}^{\alpha}\leq w_{\alpha}(\left\Vert
y\right\Vert _{p})\leq C_{1}\left\Vert y\right\Vert _{p}^{\alpha}\text{, for
}x\in\mathbb{Q}_{p}^{n}\text{.}%
\]
We now define the operator
\[
(\boldsymbol{W}_{\alpha}\varphi)(x)=\kappa{\int\limits_{\mathbb{Q}_{p}^{n}}%
}\frac{\varphi\left(  x-y\right)  -\varphi\left(  x\right)  }{w_{\alpha
}\left(  \Vert y\Vert_{p}\right)  }d^{n}y\text{, for }\varphi\in
\mathcal{D}\left(  \mathbb{Q}_{p}^{n}\right)  \text{,}%
\]
where $\kappa$ is a positive constant. The operator $\boldsymbol{W}_{\alpha}$
is pseudodifferential, more precisely, if%
\[
A_{w_{\alpha}}\left(  \xi\right)  :={\int\limits_{\mathbb{Q}_{p}^{n}}}%
\frac{1-\chi_{p}\left(  y\cdot\xi\right)  }{w_{\alpha}\left(  \Vert y\Vert
_{p}\right)  }d^{n}y,
\]
then
\[
\left(  \boldsymbol{W_{\alpha}}\varphi\right)  \left(  x\right)
=-\kappa\mathcal{F}_{\xi\rightarrow x}^{-1}\left[  A_{w_{\alpha}}\left(
\xi\right)  \mathcal{F}_{x\rightarrow\xi}\varphi\right]  \text{, for }%
\varphi\in\mathcal{D}\left(  \mathbb{Q}_{p}^{n}\right)  \text{.}%
\]
The function $A_{w_{\alpha}}\left(  \xi\right)  $ is radial (so we use the
notations $A_{w_{\alpha}}\left(  \xi\right)  =A_{w_{\alpha}}\left(  \Vert
\xi\Vert_{p}\right)  $), continuous, non-negative, $A_{w_{\alpha}}\left(
0\right)  =0$, and it satisfies
\[
C_{0}^{\prime}\left\Vert \xi\right\Vert _{p}^{\alpha-n}\leq A_{w_{\alpha}%
}(\left\Vert \xi\right\Vert _{p})\leq C_{1}^{\prime}\left\Vert \xi\right\Vert
_{p}^{\alpha-n}\text{, for }x\in\mathbb{Q}_{p}^{n}\text{,}%
\]
cf. \cite[Lemmas 3.1, 3.2, 3.3]{Ch-Z}. The operator $\boldsymbol{W}_{\alpha}$
extends to an unbounded and densely defined operator in $L^{2}\left(
\mathbb{Q}_{p}^{n}\right)  $.

\subsection{Examples}

\begin{example}
The Taibleson operator is defined as%
\[
\left(  D_{T}^{\beta}\phi\right)  \left(  x\right)  =\mathcal{F}%
_{\xi\rightarrow x}^{-1}\left(  \left\Vert \xi\right\Vert _{p}^{\beta
}\mathcal{F}_{x\rightarrow\xi}\phi\right)  \text{, with }\beta>0\text{ and
}\phi\in\mathcal{D}\left(  \mathbb{Q}_{p}^{n}\right)  .
\]
cf. \cite{Ro-Zuniga}, \cite[Section 9.2.2]{A-K-S}.
\end{example}

\begin{example}
Take $A_{\beta}\left(  \xi\right)  =\left\Vert \xi\right\Vert _{p}^{\beta
}\left\{  B-Ae^{-\left\Vert \xi\right\Vert _{p}}\right\}  $ with $B>A>0$. Then
$A_{\beta}\left(  \xi\right)  $ satisfies all the requirements announced at
the beginning of this section. In general, if $f:\mathbb{Q}_{p}^{n}%
\rightarrow\mathbb{R}_{+}$ is a radial function satisfying
\[
0<\inf_{\xi\in\mathbb{Q}_{p}^{n}}f\left(  \left\Vert \xi\right\Vert
_{p}\right)  <\sup_{\xi\in\mathbb{Q}_{p}^{n}}f\left(  \left\Vert
\xi\right\Vert _{p}\right)  <\infty,
\]
then $A_{\beta}(\left\Vert \xi\right\Vert _{p})f\left(  \left\Vert
\xi\right\Vert _{p}\right)  $ satisfies all the requirements announced at the
beginning of this section.
\end{example}

\section{Lizorkin Spaces, Eigenvalues and Eigenfunctions for $\boldsymbol{A}%
_{\beta}$ Operators}

We set $\mathcal{L}_{0}(\mathbb{Q}_{p}^{n}):=\{\varphi\in\mathcal{D}\left(
\mathbb{Q}_{p}^{n}\right)  ;\,\,\widehat{\varphi}(0)=0\}$. The $\mathbb{C}%
-$vector space $\mathcal{L}_{0}$\ is called the $p-$adic Lizorkin space of
second class. We recall that $\mathcal{L}_{0}$ is dense in $L^{2}$, cf.
\cite[Theorem 7.4.3]{A-K-S}, and that $\varphi\in\mathcal{L}_{0}\left(
\mathbb{Q}_{p}^{n}\right)  $ if and only if
\begin{equation}
\int_{\mathbb{Q}_{p}^{n}}\varphi(x)d^{n}x=0. \label{Lizorkin_property}%
\end{equation}
Consider the operator $\left(  \boldsymbol{\boldsymbol{A}}_{\beta}%
\varphi\right)  \left(  x\right)  =\mathcal{F}_{\xi\rightarrow x}^{-1}\left[
A_{\beta}\left(  \xi\right)  \mathcal{F}_{x\rightarrow\xi}\varphi\right]  $ on
$\mathcal{L}_{0}(\mathbb{Q}_{p}^{n})$, then $\boldsymbol{A}_{\beta}$ is
densely defined on $L^{2}$, and $\boldsymbol{\boldsymbol{A}}_{\beta
}:\mathcal{L}_{0}(\mathbb{Q}_{p}^{n})\rightarrow\mathcal{L}_{0}(\mathbb{Q}%
_{p}^{n})$ is a well-defined linear operator.

We set $\mathcal{L}_{0}(\mathbb{Z}_{p}^{n}):=\{\varphi\in\mathcal{L}%
_{0}(\mathbb{Q}_{p}^{n});\,$supp$\,\varphi\subseteq\mathbb{Z}_{p}^{n}\}$, and
define
\[
L_{0}^{2}\left(  \mathbb{Z}_{p}^{n},d^{n}x\right)  :=L_{0}^{2}\left(
\mathbb{Z}_{p}^{n}\right)  =\left\{  f\in L^{2}\left(  \mathbb{Z}_{p}%
^{n},d^{n}x\right)  ;\int_{\mathbb{Z}_{p}^{n}}f(x)d^{n}x=0\right\}  .
\]
Notice that, since $L_{0}^{2}\left(  \mathbb{Z}_{p}^{n}\right)  $ is the
orthogonal complement in $L^{2}\left(  \mathbb{Z}_{p}^{n}\right)  $ of the
space generated by the characteristic function of $\mathbb{Z}_{p}^{n}$, then
$L_{0}^{2}\left(  \mathbb{Z}_{p}^{n}\right)  $ is a Hilbert space.

Then $\mathcal{L}_{0}(\mathbb{Z}_{p}^{n})$ is dense in $L_{0}^{2}\left(
\mathbb{Z}_{p}^{n}\right)  $. Indeed, set
\[
\delta_{k}\left(  x\right)  :=p^{nk}\Omega\left(  p^{k}\left\Vert x\right\Vert
_{p}\right)  \text{, for }k\in\mathbb{N}.
\]
Then $\int_{\mathbb{Q}_{p}^{n}}\delta_{k}\left(  x\right)  d^{n}x=1$ for any
$k$, and take $f\in L_{0}^{2}\left(  \mathbb{Z}_{p}^{n}\right)  $, then
$f_{k}=f\ast\delta_{k}\in\mathcal{L}_{0}(\mathbb{Z}_{p}^{n})$, and $f_{k}$
$\underrightarrow{\left\Vert \cdot\right\Vert _{L^{2}}}$\ $f$.

Set
\[
\omega_{\gamma bk}(x):=p^{-\frac{n\gamma}{2}}\chi_{p}(p^{-1}k\cdot(p^{\gamma
}x-b))\Omega(\Vert p^{\gamma}x-b\Vert_{p}),
\]
where $\gamma\in\mathbb{Z}$, $b\in(\mathbb{Q}_{p}/\mathbb{Z}_{p})^{n}$,
$k=(k_{1},\dots,k_{n})$ with $k_{i}\in\{0,\dots p-1\}$ for $i=1,\ldots,n$, and
$k\neq(0,\dots,0)$.

\begin{lemma}
\label{eigenfunction} With the above notation,
\[
(\boldsymbol{A}_{\beta}\omega_{\gamma bk})(x)=\lambda_{\gamma bk}%
\omega_{\gamma bk}(x)
\]
with
\[
\lambda_{\gamma bk}=A_{\beta}({p}^{1-\gamma}).
\]
Moreover, $\int_{\mathbb{Q}_{p}^{n}}\omega_{\gamma bk}(x)d^{n}x=0$ and
$\{\omega_{\gamma bk}(x)\}_{\gamma bk}$ forms a complete orthogonal basis of
$L^{2}(\mathbb{Q}_{p}^{n},d^{n}x)$.
\end{lemma}

\begin{proof}
The result follows from Theorems 9.4.5 and 8.9.3 in \cite{A-K-S}, by using the
fact that $A_{\beta}$ satisfies $A_{\beta}(\left\Vert {p}^{\gamma}%
(-p^{-1}k+\eta)\right\Vert _{p})=A_{\beta}(\left\Vert {p}^{\gamma
-1}k\right\Vert _{p})=A_{\beta}({p}^{1-\gamma})$, for all $\eta\in
\mathbb{Z}_{p}^{n}$.
\end{proof}

\begin{remark}
\label{lambda} \noindent(i) Notice that $\boldsymbol{A}_{\beta}$ has
eigenvalues of infinity multiplicity. Now, if consider only eigenfunctions
satisfying supp $\omega_{\gamma bk}(x)\subset\mathbb{Z}_{p}^{n}$, then
necessarily $\gamma\leq0$ and $b\in p^{\gamma}\mathbb{Z}_{p}^{n}%
/\mathbb{Z}_{p}^{n}$. For $\gamma$ fixed there are only a finite number of
eigenfunctions $\omega_{\gamma bk}$ satisfying $\boldsymbol{A}_{\beta}%
\omega_{\gamma bk}=\lambda_{\gamma bk}\omega_{\gamma bk}$, i.e. the
multiplicities of the $\lambda_{\gamma bk}$ are finite. Therefore we can
number these eigenfunctions and eigenvalues in the form $\omega_{m}$,
$\lambda_{m}$ with $m\in\mathbb{N\smallsetminus}\left\{  0\right\}  $ such
that $\lambda_{m}\leq\lambda_{m^{\prime}}$ for $m\leq m^{\prime}$.

\noindent(ii) Notice that any $\omega_{m}(x)$ is orthogonal to $\Omega\left(
\left\Vert x\right\Vert _{p}\right)  $, thus $\left\{  \omega_{m}(x)\right\}
_{m\in\mathbb{N\smallsetminus}\left\{  0\right\}  }$ is not a complete
orthonormal basis of $L^{2}(\mathbb{Z}_{p}^{n},d^{n}x)$. We now recall that
$\mathcal{L}_{0}(\mathbb{Z}_{p}^{n})$ is dense in $L_{0}^{2}\left(
\mathbb{Z}_{p}^{n}\right)  $, and since the algebraic span of $\left\{
\omega_{m}(x)\right\}  _{m\in\mathbb{N\smallsetminus}\left\{  0\right\}  }$
contains $\mathcal{L}_{0}(\mathbb{Z}_{p}^{n})$, then $\left\{  \omega
_{m}(x)\right\}  _{m\in\mathbb{N\smallsetminus}\left\{  0\right\}  }$ is a
complete orthonormal basis of $L_{0}^{2}(\mathbb{Z}_{p}^{n})$.
\end{remark}

\begin{proposition}
\label{Prop1}Consider $\left(  \boldsymbol{A}_{\beta},\mathcal{L}%
_{0}(\mathbb{Z}_{p}^{n})\right)  $ and the eigenvalue problem:
\begin{equation}%
\begin{array}
[c]{ll}%
\boldsymbol{A}_{\beta}u=\lambda u, & \lambda>0,\,\,u\in\mathcal{L}%
_{0}(\mathbb{Z}_{p}^{n}).
\end{array}
\label{Spectral}%
\end{equation}
Then the function $u\left(  x\right)  =\omega_{m}(x)$\thinspace is a solution
of (\ref{Spectral}) corresponding to $\lambda=\lambda_{m}$, for $m\in
\mathbb{N\smallsetminus}\left\{  0\right\}  $. In addition, the spectrum has
the form
\[
0<\lambda_{1}\leq\lambda_{2}\leq\dots\leq\lambda_{m}\leq\dots\,\,\text{with}%
\,\,\,\lambda_{m}\uparrow+\infty,
\]
where all the eigenvalues have finite multiplicity, and $\{\omega_{m}(x)\}$,
with $m\in\mathbb{N\smallsetminus}\left\{  0\right\}  $, is a complete
orthonormal basis of $L_{0}^{2}(\mathbb{Z}_{p}^{n},d^{n}x)$.
\end{proposition}

\begin{proof}
The result follows from Lemma \ref{eigenfunction}, Remark \ref{lambda} and
(\ref{1w}).
\end{proof}

\begin{definition}
\label{Definition_zeta}We define the spectral zeta function attach to
eigenvalue problem (\ref{Spectral}) as
\[
\zeta(s;\boldsymbol{A}_{\beta},\mathcal{L}_{0}(\mathbb{Z}_{p}^{n}%
)):=\zeta(s;\boldsymbol{A}_{\beta})=\sum\limits_{m=1}^{\infty}\frac{1}%
{\lambda_{m}^{s}},\,s\in\mathbb{C}.
\]

\end{definition}

Later on we will show that $\zeta(s;\boldsymbol{A}_{\beta})$ converges if
$\operatorname{Re}(s)$ is sufficiently big, and it does not depend on the
basis $\{\omega_{m}(x)\}$ used in its computation. By abuse of language (or
following the classical literature, see \cite{Voros}), we will say that
$\zeta(s;\boldsymbol{A}_{\beta})$ is the spectral zeta function of operator
$\boldsymbol{A}_{\beta}$.

\subsection{Example\label{Example1}}

We compute $\zeta(s;D_{T}^{\beta})$. We first note that%
\[
D_{T}^{\beta}\omega_{\gamma bk}=p^{-\left(  \gamma-1\right)  \beta}%
\omega_{\gamma bk}.
\]
We now recall that if supp $\omega_{\gamma bk}\subset\mathbb{Z}_{p}^{n}$ then
$\gamma\leq0$ and $b\in p^{\gamma}\mathbb{Z}_{p}^{n}/\mathbb{Z}_{p}^{n}$. We
now take $-\gamma+1=m$, with $m\in\mathbb{N}\smallsetminus\left\{  0\right\}
$. Then $b\in p^{-m+1}\mathbb{Z}_{p}^{n}/\mathbb{Z}_{p}^{n}$ and $\lambda
_{m}=p^{m\beta}$, and the multiplicity of $\lambda_{m}$\ is equal to $\left(
p^{n}-1\right)  p^{n\left(  m-1\right)  }=p^{nm}\left(  1-p^{-n}\right)  $ for
$m\in\mathbb{N}\smallsetminus\left\{  0\right\}  $. Hence%
\[
\zeta(s;D_{T}^{\beta})=\sum\limits_{m=1}^{\infty}\frac{p^{nm}\left(
1-p^{-n}\right)  }{p^{m\beta s}}=\int\limits_{\mathbb{Q}_{p}^{n}%
\smallsetminus\mathbb{Z}_{p}^{n}}\frac{d^{n}\xi}{\left\Vert \xi\right\Vert
_{p}^{\beta s}}=\left(  1-p^{-n}\right)  \frac{p^{n-\beta s}}{1-p^{n-\beta s}%
},
\]
for $\operatorname{Re}(s)>\frac{n}{\beta}$. Then $\zeta(s;D_{T}^{\beta})$
admits a meromorphic continuation to the whole complex plane as a rational
function of $p^{-s}$ with poles \ in the set $\frac{n}{\beta}+\frac{2\pi
i\mathbb{Z}}{\beta\ln p}$.

\section{Heat traces and $p-$adic heat equations on the until ball}

From now on, $\left(  \boldsymbol{A}_{\beta},Dom\left(  \boldsymbol{A}_{\beta
}\right)  \right)  $ is given by
\begin{equation}
\left(  \boldsymbol{A}_{\beta}\varphi\right)  \left(  x\right)  =\mathcal{F}%
_{\xi\rightarrow x}^{-1}\left(  \boldsymbol{A}_{\beta}\left(  \xi\right)
\mathcal{F}_{x\rightarrow\xi}\varphi\right)  \text{ for }\varphi\in Dom\left(
\boldsymbol{A}_{\beta}\right)  =\mathcal{L}_{0}\left(  \mathbb{Z}_{p}%
^{n}\right)  . \label{Domain_of_A_beta}%
\end{equation}

\subsection{$p-$adic heat equations on the until ball}

We introduce the following function:
\[
K(x,t)=\int\limits_{\mathbb{Q}_{p}^{n}\setminus\mathbb{Z}_{p}^{n}}\chi
_{p}(-x\cdot\xi)e^{-tA_{\beta}(\xi)}d^{n}\xi,\,\,\text{for}\,\,t>0,\,\,\,x\in
\mathbb{Q}_{p}^{n}.
\]
We note that by (\ref{1w}) $e^{-tA_{\beta}(\xi)}\leq e^{-tC_{0}\left\Vert
\xi\right\Vert _{p}^{\beta}}\in L^{1}$ for $\,t>0$, which implies that
$K(x,t)$ is well-defined for$\,t>0$ and$\,\,x\in\mathbb{Q}_{p}^{n}$.

\begin{lemma}
\label{lemma2} With the above notation, the following formula holds:
\begin{multline*}
K(x,t)=\\
\left\{
\begin{array}
[c]{lll}%
\Omega\left(  \left\Vert x\right\Vert _{p}\right)  \left\{  (1-p^{-n}%
)\sum\limits_{j=1}^{ord(x)}e^{-tA_{\beta}(p^{j})}p^{nj}-p^{ord(x)n}%
e^{-tA_{\beta}(p^{ord(x)+1})}\right\}  & \text{if} & ord\left(  x\right)
\in\mathbb{N}\\
&  & \\
(1-p^{-n})\Omega\left(  \left\Vert x\right\Vert _{p}\right)  \sum
\limits_{j=1}^{\infty}e^{-tA_{\beta}(p^{j})}p^{nj} & \text{if} & ord\left(
x\right)  =+\infty
\end{array}
\right.
\end{multline*}
for any $t>0$.
\end{lemma}

\begin{proof}
Take $x=p^{ord(x)}x_{0}$, with$\,\Vert x_{0}\Vert_{p}=1$, then%
\begin{align*}
K(x,t)  &  =\sum\limits_{j=1}^{\infty}e^{-tA_{\beta}(p^{j})}\int
\limits_{\Vert\xi\Vert_{p}=p^{j}}\chi_{p}(-x\cdot\xi)d^{n}\xi\\
&  =\sum\limits_{j=1}^{\infty}e^{-tA_{\beta}(p^{j})}p^{nj}\int\limits_{\Vert
y\Vert_{p}=1}\chi_{p}(-p^{-j+ord(x)}x_{0}\cdot y)d^{n}y\\
&  =\sum\limits_{j=1}^{\infty}e^{-tA_{\beta}(p^{j})}p^{nj}\left\{
\begin{array}
[c]{ll}%
1-p^{-n} & j\leq ord(x)\\
& \\
-p^{-n} & j=ord(x)+1\\
& \\
0 & j\geq ord(x)+2.
\end{array}
\right.
\end{align*}
Then $K(x,t)=0$ for $\Vert x\Vert_{p}>1$ and $t>0$. Finally, we note that the
announced formula is valid if $x=0$.
\end{proof}

We identify $L_{0}^{2}(\mathbb{Z}_{p}^{n})$ with an isometric subspace of
$L^{2}(\mathbb{Q}_{p}^{n})$ by extending the functions of $L_{0}%
^{2}(\mathbb{Z}_{p}^{n})$ as zero outside of $\mathbb{Z}_{p}^{n}$. We define
$\left\{  T\left(  t\right)  \right\}  _{t\geq0}$ as the family of operators%
\[%
\begin{array}
[c]{ccc}%
L_{0}^{2}(\mathbb{Z}_{p}^{n}) & \rightarrow & L_{0}^{2}(\mathbb{Z}_{p}^{n})\\
&  & \\
f & \rightarrow & T\left(  t\right)  f
\end{array}
\]
with%
\[
\left(  T\left(  t\right)  f\right)  \left(  x\right)  =\left\{
\begin{array}
[c]{lll}%
f\left(  x\right)  & \text{if} & t=0\\
&  & \\
\left(  K(\cdot,t)\ast f\right)  \left(  x\right)  & \text{if} & t>0.
\end{array}
\right.
\]

\begin{lemma}
\label{Lemma_0}With the above notation the following assertions hold:

\noindent(i) operator $T\left(  t\right)  $, $t\geq0$, is a well-defined
bounded linear operator;

\noindent(ii) for $t\geq0$,%
\[
\left(  T(t)f\right)  (x)={\mathcal{F}}_{\xi\rightarrow x}^{-1}\left[
1_{\mathbb{Q}_{p}^{n}\setminus\mathbb{Z}_{p}^{n}}(\xi)e^{-tA_{\beta}(\xi
)}\widehat{f}(\xi)\right]  ,
\]
where\ $\widehat{f}(\xi)$ denotes the Fourier transform in $L^{2}%
(\mathbb{Q}_{p}^{n})$ of $f\in L_{0}^{2}(\mathbb{Z}_{p}^{n})$;

\noindent(iii) $T(t)$, for $t>0$, is a compact, self-adjoint and non-negative operator.
\end{lemma}

\begin{proof}
(i) We recall that $K\left(  \cdot,t\right)  \in L^{1}\left(  \mathbb{Q}%
_{p}^{n}\right)  $ for $t>0$, then, if $f\in L_{0}^{2}(\mathbb{Z}_{p}%
^{n})\subset L^{2}(\mathbb{Q}_{p}^{n})$, by the Young inequality,%
\[
u\left(  x,t\right)  :=\left(  K(\cdot,t)\ast f\right)  \left(  x\right)  \in
L^{2}(\mathbb{Q}_{p}^{n})\text{, for }t>0.
\]
Now, by Lemma \ref{lemma2}, supp $u\left(  x,t\right)  \subset\mathbb{Z}%
_{p}^{n}$ for $t>0$, i.e. $u\left(  x,t\right)  \in L^{2}(\mathbb{Z}_{p}^{n}%
)$, for $t>0$. Again by the Young inequality,
\begin{multline*}
\left\Vert u\left(  x,t\right)  \right\Vert _{L_{0}^{2}(\mathbb{Z}_{p}^{n}%
)}=\left\Vert u\left(  x,t\right)  \right\Vert _{L^{2}(\mathbb{Q}_{p}^{n}%
)}\leq\left\Vert K\left(  x,t\right)  \right\Vert _{L^{1}(\mathbb{Q}_{p}^{n}%
)}\left\Vert f\left(  x\right)  \right\Vert _{L^{2}(\mathbb{Q}_{p}^{n})}\\
=C(t)\left\Vert f\left(  x\right)  \right\Vert _{L_{0}^{2}(\mathbb{Z}_{p}%
^{n})}\text{, for }t>0.
\end{multline*}
We finally show that
\[
\int\nolimits_{\mathbb{Z}_{p}^{n}}u\left(  x,t\right)  d^{n}x=0\text{, for
}t>0.
\]
Indeed, for $t>0$, by using Fubini's Theorem,%
\begin{align*}
\int\nolimits_{\mathbb{Z}_{p}^{n}}u\left(  x,t\right)  d^{n}x  &
=\int\nolimits_{\mathbb{Z}_{p}^{n}}\left\{  \int\nolimits_{\mathbb{Z}_{p}^{n}%
}K(y,t)f(x-y)d^{n}y\right\}  d^{n}x\\
&  =\int\nolimits_{\mathbb{Z}_{p}^{n}}K(y,t)\left\{  \int\nolimits_{\mathbb{Z}%
_{p}^{n}}f(x-y)d^{n}x\right\}  d^{n}y\text{ \ \ (taking }z_{1}=x-y\text{,
}z_{2}=y\text{)}\\
&  =\int\nolimits_{\mathbb{Z}_{p}^{n}}K(z_{2},t)\left\{  \int
\nolimits_{\mathbb{Z}_{p}^{n}}f(z_{1})d^{n}z_{1}\right\}  d^{n}z_{2}=0.
\end{align*}

(ii) Since $f\left(  x\right)  $, $u\left(  x,t\right)  \in L^{1}%
(\mathbb{Z}_{p}^{n})\cap L^{2}(\mathbb{Z}_{p}^{n})$ for $t>0$, because
$L^{2}(\mathbb{Z}_{p}^{n})\subset L^{1}(\mathbb{Z}_{p}^{n})$, we have%
\[
\mathcal{F}_{x\rightarrow\xi}(u\left(  x,t\right)  )=1_{\mathbb{Q}_{p}%
^{n}\setminus\mathbb{Z}_{p}^{n}}(\xi)e^{-tA_{\beta}(\xi)}\widehat{f}(\xi),
\]
this last function belongs to $L^{1}(\mathbb{Q}_{p}^{n})$, indeed, by the
Cauchy-Schwarz inequality,
\begin{gather*}
\left\Vert 1_{\mathbb{Q}_{p}^{n}\setminus\mathbb{Z}_{p}^{n}}(\xi
)e^{-tA_{\beta}(\xi)}\widehat{f}(\xi)\right\Vert _{L^{1}(\mathbb{Q}_{p}^{n}%
)}\leq\left\Vert 1_{\mathbb{Q}_{p}^{n}\setminus\mathbb{Z}_{p}^{n}}%
(\xi)e^{-tA_{\beta}(\xi)}\right\Vert _{L^{2}(\mathbb{Q}_{p}^{n})}\left\Vert
\widehat{f}(\xi)\right\Vert _{L^{2}(\mathbb{Q}_{p}^{n})}\\
\leq\left\Vert e^{-tA_{\beta}(\xi)}\right\Vert _{L^{2}(\mathbb{Q}_{p}^{n}%
)}\left\Vert f(\xi)\right\Vert _{L^{2}(\mathbb{Q}_{p}^{n})}=\left\Vert
e^{-tA_{\beta}(\xi)}\right\Vert _{L^{2}(\mathbb{Q}_{p}^{n})}\left\Vert
f(\xi)\right\Vert _{L_{0}^{2}(\mathbb{Z}_{p}^{n})}<\infty
\end{gather*}
because $\int_{\mathbb{Q}_{p}^{n}}e^{-2tA_{\beta}(\xi)}d^{n}\xi\leq
\int_{\mathbb{Q}_{p}^{n}}e^{-2C_{0}t\left\Vert \xi\right\Vert _{p}^{\beta}%
}d^{n}\xi<\infty$, cf. (\ref{1w}). Finally,
\begin{align*}
\left(  T\left(  0\right)  f\right)  \left(  x\right)   &  =\int
_{\mathbb{Q}_{p}^{n}\setminus\mathbb{Z}_{p}^{n}}\chi_{p}\left(  -\xi\cdot
x\right)  \widehat{f}(\xi)d^{n}\xi\\
&  =\int_{\mathbb{Q}_{p}^{n}}\chi_{p}\left(  -\xi\cdot x\right)  \widehat
{f}(\xi)d^{n}\xi-\int_{\mathbb{Z}_{p}^{n}}\chi_{p}\left(  -\xi\cdot x\right)
\widehat{f}(\xi)d^{n}\xi\\
&  =f\left(  x\right)  -\mathcal{F}_{x\rightarrow\xi}^{-1}\left(
\Omega\left(  \left\Vert \xi\right\Vert _{p}\right)  \widehat{f}(\xi)\right)
=f\left(  x\right)  -\Omega\left(  \left\Vert x\right\Vert _{p}\right)  \ast
f(x)\\
&  =f\left(  x\right)  -\Omega\left(  \left\Vert x\right\Vert _{p}\right)
\int_{\mathbb{Z}_{p}^{n}}f(x)d^{n}x=f\left(  x\right)  .
\end{align*}

(iii) Since $T(t)$, for $t>0$, is bounded and $\left\langle
T(t)f,g\right\rangle =\left\langle f,\,T(t)g\right\rangle ,$ for $f,g\in
L^{2}(\mathbb{Z}_{p}^{n}),$ where $\left\langle \cdot,\cdot\right\rangle $
denotes the inner product of $L^{2}(\mathbb{Q}_{p}^{n})$,$\,T(t)$ is
self-adjoint for $t>0$. The compactness follows from the continuity of
operator $T(t)$.
\end{proof}

\begin{lemma}
\label{Lemma_1}The one-parameter family $\left\{  T(t)\right\}  _{t\geq0}$ of
bounded linear operators from $L_{0}^{2}\left(  \mathbb{Z}_{p}^{n}\right)  $
into itself is a contraction semigroup.
\end{lemma}

\begin{proof}
The lemma follows from the following claims.

\textbf{Claim 1.} $\left\Vert T(t)\right\Vert _{L_{0}^{2}\left(
\mathbb{Z}_{p}^{n}\right)  }\leq1$ for $t\geq0$. In addition, $\left\Vert
T(t)\right\Vert _{L_{0}^{2}\left(  \mathbb{Z}_{p}^{n}\right)  }<1$ for $t>0$.

Consider $t>0$, by Lemma \ref{Lemma_0}\ and (\ref{1w}),
\begin{gather*}
\Vert T(t)f\Vert_{L_{0}^{2}\left(  \mathbb{Z}_{p}^{n}\right)  }^{2}=\Vert
T(t)f\Vert_{L^{2}(\mathbb{Q}_{p}^{n})}^{2}=\Vert\widehat{T(t)f(\xi)}%
\Vert_{L^{2}(\mathbb{Q}_{p}^{n})}^{2}=\int\limits_{\mathbb{Q}_{p}^{n}%
\setminus\mathbb{Z}_{p}^{n}}e^{-2tA_{\beta}(\xi)}|\widehat{f}(\xi)|^{2}%
d^{n}\xi\leq\\
\int\limits_{\mathbb{Q}_{p}^{n}\setminus\mathbb{Z}_{p}^{n}}e^{-2C_{0}t\Vert
\xi\Vert_{p}^{\beta}}|\widehat{f}(\xi)|^{2}d^{n}\xi\leq\sup\limits_{\xi
\in\mathbb{Q}_{p}^{n}\setminus\mathbb{Z}_{p}^{n}}e^{-2C_{0}t\Vert\xi\Vert
_{p}^{\beta}}\int\limits_{\mathbb{Q}_{p}^{n}\setminus\mathbb{Z}_{p}^{n}%
}|\widehat{f}(\xi)|^{2}d^{n}\xi\\
<\int\limits_{\mathbb{Q}_{p}^{n}\setminus\mathbb{Z}_{p}^{n}}|\widehat{f}%
(\xi)|^{2}d^{n}\xi\leq\Vert f\Vert_{L^{2}(\mathbb{Q}_{p}^{n})}^{2}=\Vert
f\Vert_{L_{0}^{2}\left(  \mathbb{Z}_{p}^{n}\right)  }^{2},
\end{gather*}
where we used that $\underset{\xi\in\mathbb{Q}_{p}^{n}\setminus\mathbb{Z}%
_{p}^{n}}{\sup}e^{-2C_{0}t\Vert\xi\Vert_{p}^{\beta}}<1$.

\textbf{Claim 2. \ }$T(0)=I$.

\textbf{Claim 3. \ }$T(t+s)=T(t)T(s)$ for $t$, $s\geq0$.

This claim follows from Lemma \ref{Lemma_0}-(ii).

\textbf{Claim 4. }For $f\in L_{0}^{2}\left(  \mathbb{Z}_{p}^{n}\right)  $, the
function $t\rightarrow T\left(  t\right)  f$ belongs to $C\left(  \left[
0,\infty\right)  ,L_{0}^{2}\left(  \mathbb{Z}_{p}^{n}\right)  \right)  $.

Notice that since $\mathcal{L}_{0}^{2}\left(  \mathbb{Z}_{p}^{n}\right)  $ is
dense in $L_{0}^{2}\left(  \mathbb{Z}_{p}^{n}\right)  $ for $\left\Vert
\cdot\right\Vert _{L^{2}}$ norm, it is sufficient to show Claim 4 for
$f\in\mathcal{L}_{0}^{2}\left(  \mathbb{Z}_{p}^{n}\right)  $. Indeed,%
\begin{multline*}
\lim_{t\rightarrow t_{0}}\left\Vert T\left(  t\right)  f-T\left(
t_{0}\right)  f\right\Vert _{L_{0}^{2}\left(  \mathbb{Z}_{p}^{n}\right)  }%
^{2}=\lim_{t\rightarrow t_{0}}\left\Vert T\left(  t\right)  f-T\left(
t_{0}\right)  f\right\Vert _{L^{2}\left(  \mathbb{Q}_{p}^{n}\right)  }^{2}\\
=\lim_{t\rightarrow t_{0}}\left\Vert \widehat{T\left(  t\right)  f}%
-\widehat{T\left(  t_{0}\right)  f}\right\Vert _{L^{2}\left(  \mathbb{Q}%
_{p}^{n}\right)  }^{2}=\lim_{t\rightarrow t_{0}}\int\nolimits_{\mathbb{Q}%
_{p}^{n}\setminus\mathbb{Z}_{p}^{n}}\left\vert \widehat{f}(\xi)\right\vert
^{2}\left\vert e^{-tA_{\beta}(\xi)}-e^{-t_{0}A_{\beta}(\xi)}\right\vert
^{2}d^{n}\xi,
\end{multline*}
now, since $1_{\mathbb{Q}_{p}^{n}\setminus\mathbb{Z}_{p}^{n}}\left(
\xi\right)  \left\vert \widehat{f}(\xi)\right\vert ^{2}\left\vert
e^{-tA_{\beta}(\xi)}-e^{-t_{0}A_{\beta}(\xi)}\right\vert ^{2}\leq4\left\vert
\widehat{f}(\xi)\right\vert ^{2}$, which is an integrable function, by
applying the Dominated Convergence Theorem, we have $\lim_{t\rightarrow t_{0}%
}\left\Vert T\left(  t\right)  f-T\left(  t_{0}\right)  f\right\Vert
_{L_{0}^{2}\left(  \mathbb{Z}_{p}^{n}\right)  }^{2}=0$.\ 
\end{proof}

\begin{lemma}
\label{Lemma_2}The infinitesimal generator of \ semigroup $\left\{
T(t)\right\}  _{t\geq0}$ restricted to $\mathcal{L}_{0}\left(  \mathbb{Z}%
_{p}^{n}\right)  $\ agrees with $\left(  -\boldsymbol{A}_{\beta}%
,\mathcal{L}_{0}\left(  \mathbb{Z}_{p}^{n}\right)  \right)  $.
\end{lemma}

\begin{proof}
We show that
\[
\lim_{t\rightarrow0^{+}}\left\Vert \frac{T\left(  t\right)  f-f}%
{t}+\boldsymbol{A}_{\beta}f\right\Vert _{L_{0}^{2}\left(  \mathbb{Z}_{p}%
^{n}\right)  }=0,\text{ for }f\in\mathcal{L}_{0}\left(  \mathbb{Z}_{p}%
^{n}\right)  \text{.}%
\]
Indeed, \ by Lemma \ref{Lemma_0}-(ii),%
\begin{gather*}
\left\Vert \frac{T\left(  t\right)  f-f}{t}+\boldsymbol{A}_{\beta}f\right\Vert
_{L_{0}^{2}\left(  \mathbb{Z}_{p}^{n}\right)  }=\left\Vert \frac{T\left(
t\right)  f-f}{t}+\boldsymbol{A}_{\beta}f\right\Vert _{L^{2}\left(
\mathbb{Q}_{p}^{n}\right)  }\\
=\left\Vert \frac{\widehat{T\left(  t\right)  f}-\widehat{f}}{t}%
+\widehat{\boldsymbol{A}_{\beta}f}\right\Vert _{L^{2}\left(  \mathbb{Q}%
_{p}^{n}\right)  }=\left\Vert \left\{  \frac{1_{\mathbb{Q}_{p}^{n}%
\setminus\mathbb{Z}_{p}^{n}}(\xi)e^{-tA_{\beta}(\xi)}-1}{t}+A_{\beta}\left(
\xi\right)  \right\}  \widehat{f}\left(  \xi\right)  \right\Vert
_{L^{2}\left(  \mathbb{Q}_{p}^{n}\right)  }.
\end{gather*}
Now we note that
\[
\left\{  1_{\mathbb{Q}_{p}^{n}\setminus\mathbb{Z}_{p}^{n}}(\xi)e^{-tA_{\beta
}(\xi)}-1\right\}  \widehat{f}\left(  \xi\right)  =\widehat{f}\left(
\xi\right)  \left\{  e^{-tA_{\beta}(\xi)}-1\right\}  -1_{\mathbb{Z}_{p}^{n}%
}(\xi)e^{-tA_{\beta}(\xi)}\widehat{f}\left(  \xi\right)  ,
\]
and since supp $f\subset\mathbb{Z}_{p}^{n}$ then $\widehat{f}\left(  \xi
+\xi_{0}\right)  =\widehat{f}\left(  \xi\right)  $ for any $\xi_{0}%
\in\mathbb{Z}_{p}^{n}$, this fact implies that $1_{\mathbb{Z}_{p}^{n}}%
(\xi)e^{-tA_{\beta}(\xi)}\widehat{f}\left(  \xi\right)  =e^{-tA_{\beta}(\xi
)}\widehat{f}\left(  0\right)  =0$ because $f\in\mathcal{L}_{0}\left(
\mathbb{Z}_{p}^{n}\right)  $. \ Hence%
\begin{gather*}
\left\Vert \left\{  \frac{1_{\mathbb{Q}_{p}^{n}\setminus\mathbb{Z}_{p}^{n}%
}(\xi)e^{-tA_{\beta}(\xi)}-1}{t}+A_{\beta}\left(  \xi\right)  \right\}
\widehat{f}\left(  \xi\right)  \right\Vert _{L^{2}\left(  \mathbb{Q}_{p}%
^{n}\right)  }\\
=\left\Vert \frac{\left\{  e^{-tA_{\beta}(\xi)}-1\right\}  \widehat{f}\left(
\xi\right)  }{t}+A_{\beta}\left(  \xi\right)  \widehat{f}\left(  \xi\right)
\right\Vert _{L^{2}\left(  \mathbb{Q}_{p}^{n}\right)  }\\
=\left\Vert A_{\beta}(\xi)\widehat{f}\left(  \xi\right)  \left\{  1-e^{-\tau
A_{\beta}(\xi)}\right\}  \right\Vert _{L^{2}\left(  \mathbb{Q}_{p}^{n}\right)
}\text{ (for some }\tau\in\left(  0,t\right)  \text{).}%
\end{gather*}
Therefore, by the fact that $A_{\beta}(\xi)\widehat{f}\left(  \xi\right)
\in\mathcal{L}_{0}\left(  \mathbb{Q}_{p}^{n}\right)  $ and the Dominated
Convergence Theorem,
\[
\lim_{t\rightarrow0^{+}}\left\Vert \frac{T\left(  t\right)  f-f}%
{t}+\boldsymbol{A}_{\beta}f\right\Vert _{L_{0}^{2}\left(  \mathbb{Z}_{p}%
^{n}\right)  }=\lim_{t\rightarrow0^{+}}\left\Vert A_{\beta}(\xi)\widehat
{f}\left(  \xi\right)  \left\{  1-e^{-\tau A_{\beta}(\xi)}\right\}
\right\Vert _{L^{2}\left(  \mathbb{Q}_{p}^{n}\right)  }=0.
\]

\end{proof}

\begin{theorem}
\label{Theorem1}The initial value problem:%
\begin{equation}
\left\{
\begin{array}
[c]{l}%
u\left(  x,t\right)  \in C\left(  \left[  0,\infty\right)  ,Dom\left(
\boldsymbol{A}_{\beta}\right)  \right)  \cap C^{1}\left(  \left[
0,\infty\right)  ,L_{0}^{2}\left(  \mathbb{Z}_{p}^{n}\right)  \right) \\
\\
\frac{\partial u\left(  x,t\right)  }{\partial t}+\boldsymbol{A}_{\beta
}u\left(  x,t\right)  =0,\text{ }x\in\mathbb{Q}_{p}^{n}\text{, }t\in\left[
0,\infty\right) \\
\\
u\left(  x,0\right)  =\varphi\left(  x\right)  \in Dom\left(  \boldsymbol{A}%
_{\beta}\right)  ,
\end{array}
\right.  \label{Cauchy}%
\end{equation}
where $\left(  \boldsymbol{A}_{\beta},Dom\left(  \boldsymbol{A}_{\beta
}\right)  \right)  $ is given by (\ref{Domain_of_A_beta}) has a unique
solution given by $u\left(  x,t\right)  =T(t)\varphi\left(  x\right)  $.
\end{theorem}

\begin{proof}
By Lemmas \ref{Lemma_1}-\ref{Lemma_2} and the Hille-Yosida-Phillips Theorem,
see e.g. \cite[Theorem 3.4.4]{Cazenave-Haraux}, operator $\left(
-\boldsymbol{A}_{\beta},Dom\left(  \boldsymbol{A}_{\beta}\right)  \right)  $
is $m$-dissipative with dense domain \ in $L_{0}^{2}\left(  \mathbb{Z}_{p}%
^{n}\right)  $, the announced theorem now follows from \cite[Theorem 3.1.1 and
Proposition 3.4.5]{Cazenave-Haraux}.
\end{proof}

\subsection{Heat Traces}

\begin{proposition}
\label{Theorem2} Let $\{\omega_{m}\}_{m\in\mathbb{N\smallsetminus}\left\{
0\right\}  }$ be the complete orthonormal basis of $L_{0}^{2}(\mathbb{Z}%
_{p}^{n})$ as above. Then
\[
K(x-y,t)=\sum\limits_{m=1}^{\infty}e^{-\lambda_{m}t}\omega_{m}(x)\overline
{\omega_{m}(y)}%
\]
where the convergence is uniform on $\mathbb{Z}_{p}^{n}\times\mathbb{Z}%
_{p}^{n}\times\lbrack\epsilon,\infty),$ for every $\epsilon>0.$
\end{proposition}

\begin{proof}
By applying Hilbert-Schmidt Theorem to $T(1)$, see e.g. \cite[Theorem
VI.16]{Reed-Simon}, which is self-adjoint and compact, cf. Lemma \ref{Lemma_0}
(iii), there exists a complete orthonormal basis $\{\phi_{m}\}$,
$m\in\mathbb{N\smallsetminus}\left\{  0\right\}  $, of $L_{0}^{2}%
(\mathbb{Z}_{p}^{n})$ consisting of eigenfunctions of $T(1)$, let $\{\mu
_{m}\}$, $m\in\mathbb{N\smallsetminus}\left\{  0\right\}  $, the sequence of
corresponding eigenvalues. In addition, $\mu_{m}\rightarrow0$ as
$m\rightarrow\infty$. By using the fact $\{T(t)\}_{t\geq0}$ form a semigroup,
$T(\frac{l}{k})\phi_{m}=\mu_{m}^{l/k}\phi_{m}$, for every positive rational
number $\frac{l}{k}$. Now from the continuity of $\{T(t)\}_{t\geq0}$ we get
\[
T(t)\phi_{m}=\mu_{m}^{t}\phi_{m}\text{, for }t\in\mathbb{R}_{+}\text{.}%
\]
We note that $\mu_{m}>0$ for every $m$ since
\[
\phi_{m}=\lim\limits_{t\rightarrow0^{+}}T(t)\phi_{m}=\phi_{m}\lim
\limits_{t\rightarrow0^{+}}\mu_{m}^{t}%
\]
implies that $\lim\limits_{t\rightarrow0^{+}}\mu_{m}^{t}=1$ because $\phi
_{m}\neq0$. Hence $\mu_{m}=e^{-\lambda_{m}}$, with$\,\lambda_{m}>0$, because
$\left\Vert T(t)\right\Vert _{L_{0}^{2}\left(  \mathbb{Z}_{p}^{n}\right)  }<1$
for $t>0$, cf. Lemma \ref{Lemma_1} (i),\ implies that $\mu_{m}<1$ and
$\lim\limits_{m\rightarrow\infty}\lambda_{m}=\infty$, since $\lim
\limits_{m\rightarrow\infty}\mu_{m}=0$.

By using Mercer's Theorem, see e.g. \cite[and the references therein]%
{Dodziuk}, \cite{Riesz},
\begin{equation}
K(x-y,t)=\sum\limits_{m=1}^{\infty}e^{-\lambda_{m}t}\phi_{m}(x)\overline
{\phi_{m}(y)}. \label{Formula_Z_+}%
\end{equation}
Now, since $T(t)\phi_{m}(x)=e^{-\lambda_{m}t}\phi_{m}(x)$ is a solution of
problem (\ref{Cauchy}) with initial datum $\phi_{m}$, cf. Theorem
\ref{Theorem1}, and
\[
-\lambda_{m}e^{-\lambda_{m}t}\phi_{m}(x)=\frac{\partial}{\partial
t}(e^{-\lambda_{m}t}\phi_{m}(x))=-\boldsymbol{A}_{\beta}\left(  e^{-\lambda
_{m}t}\phi_{m}(x)\right)  =-e^{-\lambda_{m}t}\boldsymbol{A}_{\beta}\phi
_{m}(x),
\]
then $\phi_{m}(x)$ is an eigenfunction of $\boldsymbol{A}_{\beta}$ with supp
$\phi_{m}\subset\mathbb{Z}_{p}^{n}$. Now, by using that $\boldsymbol{A}%
_{\beta}\omega_{m}=\lambda_{m}\omega_{m}$, see Proposition \ref{Prop1}, we get
that $u=e^{-\lambda_{m}t}\omega_{m}$ is a solution of the following boundary
value problem:
\[
\left\{
\begin{array}
[c]{ll}%
\frac{\partial u\left(  x,t\right)  }{\partial t}=-\boldsymbol{A}_{\beta
}u\left(  x,t\right)  , & u\left(  x,t\right)  \in L_{0}^{2}\left(
\mathbb{Z}_{p}^{n}\right)  \text{, for }t\geq0\\
& \\
u\left(  x,0\right)  =\omega_{m}\left(  x\right)  , & \omega_{m}\left(
x\right)  \in\mathcal{L}_{0}\left(  \mathbb{Z}_{p}^{n}\right)  .
\end{array}
\right.
\]
Then, by Theorem \ref{Theorem1}, the above problem has a unique solution,
which implies that
\[
u\left(  x,t\right)  =T(t)\omega_{m}\left(  x\right)  =e^{-\lambda_{m}t}%
\omega_{m},
\]
hence we can replace $\{\phi_{m}\}$ by $\{\omega_{m}\}$ in (\ref{Formula_Z_+}).
\end{proof}

In the next result, we will use the classical notation $e^{-t\boldsymbol{A}%
_{\beta}}$ for operator $T(t)$ to emphasize the dependency on operator
$\boldsymbol{A}_{\beta}$.

\begin{theorem}
\label{Theorem2A}The operator $e^{-t\boldsymbol{A}_{\beta}}$, for $t>0$, is
trace class and it verifies:

\noindent(i)
\begin{equation}
Tr\left(  e^{-t\boldsymbol{A}_{\beta}}\right)  =\sum\limits_{m=1}^{\infty
}e^{-\lambda_{m}t}=\int\limits_{\mathbb{Q}_{p}^{n}\smallsetminus\mathbb{Z}%
_{p}^{n}}e^{-tA_{\beta}(\xi)}d^{n}\xi, \label{TRACE}%
\end{equation}
for $t>0$;

\noindent(ii) there exist \ positive constants $C$, $C^{\prime}$ such that
$\ $%
\[
Ct^{-\frac{n}{\beta}}\leq Tr\left(  e^{-t\boldsymbol{A}_{\beta}}\right)  \leq
C^{\prime}t^{-\frac{n}{\beta}},
\]
for $t>0$.
\end{theorem}

\begin{proof}
By Proposition \ref{Theorem2} and the definition of $K(x,t)$:%
\begin{equation}
K(0,t)=\int\limits_{\mathbb{Q}_{p}^{n}\smallsetminus\mathbb{Z}_{p}^{n}%
}e^{-tA_{\beta}(\xi)}d^{n}\xi=\sum\limits_{m=1}^{\infty}e^{-\lambda_{m}%
t}\left\vert \omega_{m}(x)\right\vert ^{2}, \label{Eq1}%
\end{equation}
for $t>0$. By using the Dominated Convergence Theorem and the fact that
$\sum_{m}e^{-\lambda_{m}t}$ converges for $t>0$, we can integrate both sides
of (\ref{Eq1}) with respect to the variable $x$ over $\mathbb{Z}_{p}^{n}$ to
get
\begin{equation}
\int\limits_{\mathbb{Q}_{p}^{n}\smallsetminus\mathbb{Z}_{p}^{n}}e^{-tA_{\beta
}(\xi)}d^{n}\xi=\sum\limits_{m=1}^{\infty}e^{-\lambda_{m}t}\text{, for }t>0.
\label{Eq2}%
\end{equation}
We recall that
\begin{equation}
e^{-C_{1}t\Vert\xi\Vert_{p}^{\beta}}\leq e^{-tA_{\beta}(\xi)}\leq
e^{-C_{0}t\Vert\xi\Vert_{p}^{\beta}}, \label{Eq4}%
\end{equation}
cf. (\ref{1w}), and that $e^{-Ct\Vert\xi\Vert_{p}^{\beta}}\in L^{1}$, for
$t>0$ and for any positive constant $C$, then the series on the right-hand
side of (\ref{Eq2}) converges. Now%
\begin{align*}
Tr\left(  e^{-t\boldsymbol{A}_{\beta}}\right)   &  =\sum\limits_{m=1}^{\infty
}\left\langle e^{-t\boldsymbol{A}_{\beta}}\omega_{m},\omega_{m}\right\rangle
=\sum\limits_{m=1}^{\infty}e^{-\lambda_{m}t}\left\Vert \omega_{m}\right\Vert
_{L^{2}}^{2}\\
&  =\sum\limits_{m=1}^{\infty}e^{-\lambda_{m}t}<\infty\text{, for }t>0\text{,}%
\end{align*}
i.e. $e^{-t\boldsymbol{A}_{\beta}}$ is trace class and the formula announced
in (i) holds. The estimations for $Tr\left(  e^{-t\boldsymbol{A}_{\beta}%
}\right)  $ follows from (\ref{Eq4}), by using $\int_{\mathbb{Q}_{p}^{n}%
}e^{-Ct\Vert\xi\Vert_{p}^{\beta}}d^{n}\xi\leq Dt^{-\frac{n}{\beta}}$ for $t>0$.
\end{proof}

\section{Analytic Continuation of Spectral Zeta Functions}

\begin{remark}
\label{nota_holomorfia}(i) We set for $a>0$, $a^{s}:=e^{s\ln a}$. Then $a^{s}$
becomes a holomorphic function on $\operatorname{Re}(s)>0$.

\noindent(ii) We recall the following fact, see e.g. \cite[Lemma 5.3.1]%
{Igusa}. Let $\left(  X,d\mu\right)  $ denote a measure space, $U$ a non-empty
open subset of $\mathbb{C}$, \ and $f:X\times U\rightarrow\mathbb{C}$ a
measurable function. Assume that: (1) if $\mathcal{C}$ is a compact subset of
$U$, there exists an integrable function $\phi_{C}\geq0$ on $X$ satisfying
$\left\vert f\left(  \xi,s\right)  \right\vert \leq\phi_{C}\left(  \xi\right)
$ for all $\left(  \xi,s\right)  \in X\times\mathcal{C}$; (2) $f\left(
\xi,\cdot\right)  $ is holomorphic on $U$ for every $x$ in $X$. Then $\int
_{X}f\left(  \xi,s\right)  d\mu$ is a holomorphic function on $U$.
\end{remark}

\begin{proposition}
\label{Prop2} The spectral zeta function for $\boldsymbol{A}_{\beta}$ is a
holomorphic function on $\operatorname{Re}(s)>\frac{n}{\beta}$ and satisfies%
\begin{equation}
\zeta(s;\boldsymbol{A}_{\beta})=\int\limits_{\mathbb{Q}_{p}^{n}\smallsetminus
\mathbb{Z}_{p}^{n}}\frac{d^{n}\xi}{A_{\beta}^{s}(\xi)}\text{ for
}\operatorname{Re}(s)>\frac{n}{\beta}\text{.} \label{Formula_Zeta}%
\end{equation}
In particular $\zeta(s;\boldsymbol{A}_{\beta})$ does not depend on the basis
of $L_{0}^{2}(\mathbb{Z}_{p}^{n})$ used in Definition \ref{Definition_zeta}.
\end{proposition}

\begin{proof}
By using Proposition \ref{Prop1} and Remark \ref{lambda}, the eigenvalues have
the form $A_{\beta}({p}^{1-\gamma})$, with $\gamma\leq0$, and the
corresponding multiplicity is the cardinality of $p^{\gamma}\mathbb{Z}_{p}%
^{n}/\mathbb{Z}_{p}^{n}$ times the cardinality of the set of $k$'s, i.e.
$p^{-\gamma n}\left(  p^{n}-1\right)  $, then
\begin{align*}
\zeta(s;\boldsymbol{A}_{\beta})  &  =\sum\limits_{\gamma\leq0}\frac{p^{-\gamma
n}\left(  p^{n}-1\right)  }{A_{\beta}^{s}(p^{1-\gamma})}=\sum\limits_{m=1}%
^{\infty}\frac{p^{mn}\left(  1-p^{-n}\right)  }{A_{\beta}^{s}(p^{m})}%
=\sum\limits_{m=1}^{\infty}\int\limits_{\left\Vert \xi\right\Vert _{p}=p^{m}%
}\frac{d^{n}\xi}{A_{\beta}^{s}(\left\Vert \xi\right\Vert _{p})}\\
&  =\int\limits_{\mathbb{Q}_{p}^{n}\smallsetminus\mathbb{Z}_{p}^{n}}%
\frac{d^{n}\xi}{A_{\beta}^{s}(\xi)},
\end{align*}
and by (\ref{1w}),
\[
|\zeta(s;\boldsymbol{A}_{\beta})|\leq\frac{\left(  1-p^{-n}\right)
}{C^{Re(s)}}\sum\limits_{m=1}^{\infty}p^{m\left(  n-\beta Re(s)\right)
}<\infty\text{ for }Re(s)>\frac{n}{\beta}.
\]
To establish the holomorphy on $\operatorname{Re}(s)>\frac{n}{\beta}$ we use
Remark \ref{nota_holomorfia} (ii). Take $X=\mathbb{Q}_{p}^{n}\smallsetminus
\mathbb{Z}_{p}^{n}$, $d\mu=d^{n}\xi$, $U=\left\{  s\in\mathbb{C}%
;\operatorname{Re}(s)>\frac{n}{\beta}\right\}  $ and $f(\xi,s)=A_{\beta}%
^{-s}(\Vert\xi\Vert_{p})$. We now verify the two conditions established in
Remark \ref{nota_holomorfia} (ii). Take $\mathcal{C}$ a compact subset of $U$,
by (\ref{1w})
\[
\left\vert \frac{1}{A_{\beta}^{s}(\Vert\xi\Vert_{p})}\right\vert \leq\frac
{1}{C^{\operatorname{Re}(s)}\left\Vert \xi\right\Vert _{p}^{\beta
\operatorname{Re}(s)}},
\]
where $C$ is a positive constant. Since $\operatorname{Re}(s)$\ belongs to a
compact subset of
\[
\left\{  s\in\mathbb{R};\operatorname{Re}(s)>\frac{n}{\beta}\right\}  ,
\]
we may assume without loss of generality that $\operatorname{Re}(s)\in\left[
\gamma_{0},\gamma_{1}\right]  $ with $\gamma_{0}>\frac{n}{\beta}$, then
\[
\frac{1}{C^{\operatorname{Re}(s)}\left\Vert \xi\right\Vert _{p}^{\beta
\operatorname{Re}(s)}}\leq B(\mathcal{C})\frac{1}{\left\Vert \xi\right\Vert
_{p}^{\beta\gamma_{0}}}\in L^{1},
\]
where $B(\mathcal{C})$ is a positive constant. Condition (2) in Remark
\ref{nota_holomorfia} (ii), follows from Remark \ref{nota_holomorfia} (i), by
noting that $\left(  A_{\beta}(\Vert\xi\Vert_{p})\right)  ^{-s}=\exp\left(
-s\ln A_{\beta}(\Vert\xi\Vert_{p})\right)  $ with $A_{\beta}(\Vert\xi\Vert
_{p})>0$ for $\Vert\xi\Vert_{p}>1$.
\end{proof}

\begin{remark}
We notice that formula (\ref{Formula_Zeta}) can be obtained by taking the
Mellin transform in (\ref{TRACE}). Indeed,%
\begin{equation}
\int\limits_{0}^{\infty}\left\{  \int\limits_{\mathbb{Q}_{p}^{n}%
\smallsetminus\mathbb{Z}_{p}^{n}}e^{-tA_{\beta}(\Vert\xi\Vert_{p})}%
t^{s-1}d^{n}\xi\right\}  dt=\int\limits_{0}^{\infty}\left\{  \sum
\limits_{m=1}^{\infty}e^{-\lambda_{m}t}t^{s-1}\right\}  dt=\Gamma\left(
s\right)  \zeta(s;\boldsymbol{A}_{\beta}),\nonumber
\end{equation}
for $\operatorname{Re}(s)>1$, where $\Gamma\left(  s\right)  $ denotes the
Archimedean Gamma function. Now, By changing variables as $y=A_{\beta}%
(\Vert\xi\Vert_{p})t$, with $\xi$ fixed, we have%
\[
\zeta(s;\boldsymbol{A}_{\beta})=\int\limits_{\mathbb{Q}_{p}^{n}\smallsetminus
\mathbb{Z}_{p}^{n}}\frac{d^{n}\xi}{A_{\beta}^{s}(\Vert\xi\Vert_{p})}\text{ for
}\operatorname{Re}(s)>\max\left\{  1,\frac{n}{\beta}\right\}  .
\]

\end{remark}

\begin{lemma}
\label{lemma5}$\zeta(s;\boldsymbol{A}_{\beta})$ has a simple pole at
$s=\frac{n}{\boldsymbol{\beta}}$.
\end{lemma}

\begin{proof}
Set $\sigma\in\mathbb{R}_{+}$, since%
\[
\zeta(\sigma;\boldsymbol{A}_{\beta})\leq\frac{1}{C_{0}}\int\limits_{\mathbb{Q}%
_{p}^{n}\smallsetminus\mathbb{Z}_{p}^{n}}\frac{d^{n}\xi}{\left\Vert
\xi\right\Vert _{p}^{\boldsymbol{\beta}\sigma}}=\frac{\left(  1-p^{-n}\right)
p^{-\beta\sigma+n}}{C_{0}\left(  1-p^{-\boldsymbol{\beta}\sigma+n}\right)
}\text{ for }\sigma>\frac{n}{\boldsymbol{\beta}},
\]
we have
\begin{equation}
\lim_{\sigma\rightarrow\frac{n}{\boldsymbol{\beta}}}\left(
1-p^{-\boldsymbol{\beta}\sigma+n}\right)  \zeta(\sigma;\boldsymbol{A}_{\beta
})>0. \label{part_I}%
\end{equation}
The assertion follows from (\ref{part_I}), by using the fact that
$1-p^{-\boldsymbol{\beta}\sigma+n}$ has a simple zero at $\frac{n}%
{\boldsymbol{\beta}}$. Indeed,%
\begin{multline*}
1-p^{-\boldsymbol{\beta}\sigma+n}=1-\exp\left\{  \left(  -\boldsymbol{\beta
}\sigma+n\right)  \ln p\right\} \\
=\left\{  \boldsymbol{\beta}\ln p\right\}  \left(  \sigma-\frac{n}%
{\boldsymbol{\beta}}\right)  +O(\left(  \sigma-\frac{n}{\boldsymbol{\beta}%
}\right)  ^{2}),
\end{multline*}
where $O$ is an analytic function satisfying $O\left(  0\right)  =0$.
\end{proof}

\begin{theorem}
\label{Theorem3}The spectral zeta function $\zeta(s;\boldsymbol{A}_{\beta})$
satisfies the following:

\noindent(i) $\zeta(s;\boldsymbol{A}_{\beta})$\ is a holomorphic function on
$\operatorname{Re}(s)>\frac{n}{\boldsymbol{\beta}}$, and on this region is
given by formula (\ref{Formula_Zeta});

\noindent(ii) $\zeta(s;\boldsymbol{A}_{\beta})$ has a simple pole \ at
$s=\frac{n}{\boldsymbol{\beta}}$, however, this pole is not necessarily unique;

\noindent(iii) set $N(T):=\sum_{\lambda_{m}\leq T}1$, for $T\geq0$, then
$N(T)=O\left(  T^{\frac{n}{\boldsymbol{\beta}}}\right)  $.
\end{theorem}

\begin{proof}
(i) See Proposition \ref{Prop2}. (ii) The first part was established in Lemma
\ref{lemma5}. Take $\boldsymbol{A}_{\beta}$ to be the Taibleson operator
$D_{T}^{\boldsymbol{\beta}}$, then $\zeta(s;\boldsymbol{D}_{T}^{\beta})$ has a
meromorphic continuation to the whole complex plane as a rational function of
$p^{-s}$ having poles in the set $\frac{n}{\beta}+\frac{2\pi i\mathbb{Z}%
}{\beta\ln p}$, see Example \ref{Example1}. (iii) The result follows from the
formulas%
\[
\lambda_{m}=A_{\beta}(p^{m})\text{ and mult}\left(  \lambda_{m}\right)
=p^{nm}\left(  1-p^{-n}\right)  \text{, for }m\in\mathbb{N\smallsetminus
}\left\{  0\right\}  .
\]

\end{proof}

\begin{remark}
\label{Nota3}The fact that $\zeta(s;\boldsymbol{A}_{\beta})$\ may have several
poles on the line $\operatorname{Re}(s)=\frac{n}{\beta}$ prevent us of using
the classical Ikehara Tauberian Theorem \ to obtain the asymptotic behavior of
$N(T)$, see e.g. \cite[Appendix A]{Chamber-Loir-Tschinkel}, \cite[Chapter 2,
Section 14]{Shubin}. Any way we expect that

\begin{conjecture}
\label{Conjecture}$N(T)\sim CT^{\frac{n}{\beta}}$, for some suitable positive
constant $C$.\medskip
\end{conjecture}
\end{remark}


\begin{thebibliography}{99}                                                                                               %


\bibitem {Alberverio et al}Albeverio Sergio, Karwowski Witold, Yasuda Kumi,
Trace formula for $p$-adics, Acta Appl. Math. 71 (2002), no. 1, 31--48.

\bibitem {A-K-S}Albeverio S., Khrennikov A. Yu., Shelkovich V. M., Theory of
$p$-adic distributions: linear and nonlinear models. Cambridge University
Press, 2010.

\bibitem {Arendt et al}Arendt W., Nittka R., Peter W., Steiner F., Weyl's law:
Spectral properties of the laplacian in mathematical physics. In Mathematical
Analysis of Evolution, Information, and Complexity pp. 1-71, (2009).

\bibitem {Av-4}Avetisov V. A., Bikulov A. Kh., Osipov V. A., $p$-adic
description of characteristic relaxation in complex systems, J. Phys. A 36
(2003), no. 15, 4239--4246.

\bibitem {Av-5}Avetisov V. A., Bikulov A. H., Kozyrev S. V., Osipov V. A.,
$p$-adic models of ultrametric diffusion constrained by hierarchical energy
landscapes, J. Phys. A 35 (2002), no. 2, 177--189.

\bibitem {Burnol}Burnol Jean-François, Scattering on the $p$-adic field and a
trace formula, Internat. Math. Res. Notices 2000, no. 2, 57--70.

\bibitem {Casas-Zuniga}Casas-Sánchez O. F., Zúñiga-Galindo W. A., $p$-adic
elliptic quadratic forms, parabolic-type pseudodifferential equations with
variable coefficients and Markov processes, $p$-Adic Numbers Ultrametric Anal.
Appl. 6 (2014), no. 1, 1--20.

\bibitem {Cazenave-Haraux}Cazenave Thierry, Haraux Alain, An introduction to
semilinear evolution equations. Oxford University Press, 1998.

\bibitem {Ch-Z}Chacón-Cortés L.F., Zúñiga-Galindo W. A., Nonlocal Operators,
Parabolic-Type Equations, and Ultrametric Random Walks, J. Math. Phys. 54,
113503 (2013). Erratum 55 (2014), no. 10, 109901, 1 pp.

\bibitem {Ch-Z-1}Chacón-Cortés L. F., Zúñiga-Galindo W. A., Non-local
operators, non-Archimedean parabolic-type equations with variable coefficients
and Markov processes, Publ. Res. Inst. Math. Sci. 51 (2015), no. 2, 289--317.

\bibitem {Chamber-Loir-Tschinkel}Chambert-Loir Antoine, Tschinkel Yuri, Igusa
integrals and volume asymptotics in analytic and adelic geometry, Confluentes
Math. 2 (2010), no. 3, 351--429.

\bibitem {Connes}Connes A., Trace formula in non-commutative geometry and the
zeros of the Riemann zeta function, Selecta Math (N.S.) 5 (1999), 29--106.

\bibitem {Denninger}Deninger Christopher, On the nature of the "explicit
formulas\textquotedblright\ in analytic number theory---a simple example.
Number theoretic methods (Iizuka, 2001), 97--118, Dev. Math., 8, Kluwer Acad.
Publ., Dordrecht, 2002.

\bibitem {Dodziuk}Dodziuk Jozef, Eigenvalues of the Laplacian and the heat
equation, Amer. Math. Monthly 88 (1981), no. 9, 686--695.

\bibitem {Dra-Kh-K-V}Dragovich B., Khrennikov A. Yu., Kozyrev S. V., Volovich,
I. V., On $p$-adic mathematical physics, p-Adic Numbers Ultrametric Anal.
Appl. 1 (2009), no. 1, 1--17.

\bibitem {Gilkey}Gilkey Peter B., Invariance theory, the heat equation, and
the Atiyah-Singer index theorem. Second edition. Studies in Advanced
Mathematics. CRC Press, Boca Raton, FL, 1995.

\bibitem {Igusa}Igusa J.-I., An introduction \ to the theory of local zeta
functions, AMS/IP Studies in Advanced Mathematics, 2000.

\bibitem {Koch}Kochubei Anatoly N., Pseudo-differential equations and
stochastics over non-Archimedean fields. Marcel Dekker, Inc., New York, 2001.

\bibitem {Lal-Lapidius}Lal Nishu, Lapidus Michel L., Hyperfunctions and
spectral zeta functions of Laplacians on self-similar fractals, J. Phys. A 45
(2012), no. 36, 365205, 14 pp.

\bibitem {Leic}Leichtnam Eric, Scaling group flow and Lefschetz trace formula
for laminated spaces with $p$-adic transversal. Bull. Sci. Math. 131 (2007),
no. 7, 638--669.

\bibitem {Leic-1}Leichtnam Eric, On the analogy between arithmetic geometry
and foliated spaces, Rend. Mat. Appl. (7) 28 (2008), no. 2, 163--188.

\bibitem {Minak}Minakshisundaram S., A generalization of Epstein zeta
functions. With a supplementary note by Hermann Weyl, Canadian J. Math. 1,
(1949). 320--327.

\bibitem {Minak2}Minakshisundaram S., Eigenfunctions on Riemannian manifolds,
J. Indian Math. Soc. (N.S.) 17 (1953), 159--165 (1954).

\bibitem {Minak-P}Minakshisundaram S., Pleijel Å, Some properties of the
eigenfunctions of the Laplace-operator on Riemannian manifolds, Canadian J.
Math. 1, (1949). 242--256.

\bibitem {Reed-Simon}Reed Michael, Simon Barry, Methods of modern mathematical
physics. I. Functional analysis. Academic Press, New York-London, 1972.

\bibitem {Riesz}Riesz Frigyes, Sz.-Nagy Béla, Functional analysis. Dover
Publications, Inc., New York, 1990.

\bibitem {Ro-Zuniga}Rodrí­guez-Vega J. J., Zúñiga-Galindo W. A., Taibleson
operators, $p$-adic parabolic equations and ultrametric diffusion, Pacific J.
Math. 237 (2008), no. 2, 327--347.

\bibitem {Shubin}Shubin M. A., Pseudodifferential operators and spectral
theory. Springer-Verlag, Berlin, 2001.

\bibitem {S-T}Steinhurst Benjamin A., Teplyaev Alexander, Existence of a
meromorphic extension of spectral zeta functions on fractals, Lett. Math.
Phys. 103 (2013), no. 12, 1377--1388.

\bibitem {Taibleson}Taibleson M. H., Fourier analysis on local fields,
Princeton University Press, 1975.

\bibitem {T-Z}Torba S. M., Zú\~niga-Galindo W. A., Parabolic type equations
and Markov stochastic processes on adeles, J. Fourier Anal. Appl. 19 (2013),
no. 4, 792--835.

\bibitem {V-V-Z}Vladimirov V. S., Volovich I. V., Zelenov E. I., $p$-adic
analysis and mathematical physics, World Scientific, 1994.

\bibitem {Voros}Voros André, Spectral zeta functions. Zeta functions in
geometry (Tokyo, 1990), 327--358, Adv. Stud. Pure Math., 21, Kinokuniya,
Tokyo, 1992.

\bibitem {Yasuda}Yasuda Kumi, Trace formula on the p-adic upper half-plane, J.
Funct. Anal. 216 (2004), no. 2, 422--454.

\bibitem {Zuniga1}Zú\~niga-Galindo W. A., The non-Archimedean stochastic heat
equation driven by Gaussian noise, J. Fourier Anal. Appl. 21 (2015), no. 3, 600--627.

\bibitem {Zuniga2}Zúñiga-Galindo W. A., Parabolic equations and Markov
processes over $p$-adic fields, Potential Anal. 28 (2008), no. 2, 185--200.
\end{thebibliography}
\end{document}